\documentclass[10pt, a4paper, leqno]{amsart}

\usepackage[abbrev]{amsrefs}
\usepackage{backref}
\usepackage{lmodern}
\usepackage{microtype}
\usepackage{enumerate}
\usepackage{paralist}

\newcommand{\bfrac}[2]{{#1/#2}}

\usepackage[top=4cm,bottom=4cm,right=3cm,left=3cm]{geometry}

\newtheorem{theorem}{Theorem}
\newtheorem{proposition}{Proposition}[section]
\newtheorem{claim}{Claim}
\newtheorem{lemma}[proposition]{Lemma}
\newtheorem{corollary}[proposition]{Corollary}
\theoremstyle{definition}
\newtheorem{remark}{Remark}[section]
\numberwithin{equation}{section}

\newcommand{\N}{{\mathbb N}}
\newcommand{\R}{{\mathbb R}}

\newcommand{\abs}[1]{\lvert #1 \rvert}
\newcommand{\bigabs}[1]{\bigl\lvert #1 \bigr\rvert}

\newcommand{\norm}[1]{\lVert #1 \rVert}

\newcommand{\weakto}{\rightharpoonup}
\newcommand{\st}{\;:\;}

\newcommand{\dif}{\,\mathrm{d}}

\newenvironment{proofclaim}[1][Proof of the claim]{\begin{proof}[#1]}{\end{proof}}

\allowdisplaybreaks

\title[Existence of groundstates]{Existence of groundstates for a class\\ of nonlinear Choquard equations}
\author{Vitaly Moroz}
\address{Swansea University\\ Department of Mathematics\\ Singleton Park\\
Swansea\\ SA2~8PP\\ Wales, United Kingdom}	
\email{V.Moroz@swansea.ac.uk}

\author{Jean Van Schaftingen}
\address{Universit\'e Catholique de Louvain\\ Institut de Recherche en Math\'ematique et Phy\-sique\\ Chemin du Cyclotron 2 bte L7.01.01\\ 1348 Louvain-la-Neuve \\ Belgium}
\email{Jean.VanSchaftingen@uclouvain.be}



\date{\today}

\begin{document}

\begin{abstract}
We prove the existence of a nontrivial solution \(u \in H^1 (\R^N)\) to the nonlinear Choquard equation
\[
 - \Delta u + u = \bigl(I_\alpha \ast F (u)\bigr) F' (u) \quad \text{in \(\R^N\),}
\]
where \(I_\alpha\) is a Riesz potential, under almost necessary conditions on the nonlinearity \(F\)
in the spirit of Berestycki and Lions.
This solution is a groundstate; if moreover \(F\) is even and monotone on \((0,\infty)\), 
then \(u\) is of constant sign and radially symmetric.
\end{abstract}

\maketitle

\tableofcontents

\section{Introduction}

We consider the problem
\begin{equation}
\tag{$\mathcal{P}$}
\label{problemMain}
 - \Delta u + u = \bigl(I_\alpha \ast F (u)\bigr) f (u) \quad \text{in \(\R^N\)},
\end{equation}
where \(N \ge 3\), \(\alpha \in (0, N)\), \(I_\alpha : \R^N \to \R\) is the Riesz potential defined for every \(x \in \R^N \setminus \{0\}\) by
\[
  I_\alpha(x)=\frac{\Gamma(\tfrac{N-\alpha}{2})}
                   {\Gamma(\tfrac{\alpha}{2})\pi^{N/2}2^{\alpha} \abs{x}^{N-\alpha}},
\]
\(F\in C^1(\R;\R)\) and \(f:=F'\).
Solutions of \eqref{problemMain} are formally critical points of the functional defined by
\[
 \mathcal{I} (u) = \frac{1}{2} \int_{\R^N} \abs{\nabla u}^2 + \abs{u}^2 - \frac{1}{2} \int_{\R^N} \bigl(I_\alpha \ast F (u)\bigr) F (u).
\]
We are interested in the existence of solutions to \eqref{problemMain}.

Problem \eqref{problemMain} is a semilinear elliptic problem with a nonlocal nonlinearity and covers in particular for \(N = 3\), \(\alpha = 2\) and \(F (s) = \frac{s^2}{2}\) the Choquard--Pekar equation \citelist{\cite{Pekar-1954}\cite{Lieb-1977}}, also known as the stationary Hartree equation or the Newton--Schr\"odinger equation \cite{Moroz-Penrose-Tod-1998}.
In this case the existence of solutions was proved by variational methods by E.\thinspace H.\thinspace Lieb, P.-L.\thinspace Lions and G.\thinspace Menzala \citelist{\cite{Lieb-1977}\cite{Lions-1980}\cite{Menzala-1980}} and also by ordinary differential equations techniques \citelist{\cite{Tod-Moroz-1999}\cite{Moroz-Penrose-Tod-1998}\cite{Choquard-Stubbe-Vuffray-2008}}.
In the more general case \(F (s) = \frac{s^p}{p}\), problem \eqref{problemMain} is known to have a solution if and only if
\(\frac{N + \alpha}{N - 2} < p < \frac{N + \alpha}{N}\) \citelist{\cite{Ma-Zhao-2010}*{p.\thinspace 457}\cite{MorozVanSchaftingenGround}*{theorem~1}} (see also \cite{GenevVenkov2012}*{Lemma 2.7}).

The existence results up until now were only available when the nonlinearity \(F\) is \emph{homogeneous}.
This situation contrasts with the striking existence result for the corresponding local problem
\begin{equation}
\label{problemBL}
 - \Delta u + u = g (u)\quad \text{in \(\R^N\)},
\end{equation}
which can be considered as a limiting problem of \eqref{problemMain} when \(\alpha \to 0\), with \(g = F f\).
H.\thinspace Berestycki and P.-L.\thinspace Lions \cite{BerestyckiLions1983}*{theorem 1} have proved that \eqref{problemBL} has a nontrivial solution if nonlinearity \(g\in C(\R;\R)\) satisfies the assumptions
\begin{gather*}
\tag{$g_1$}\label{assumptWeakg} \text{there exists \(C > 0\) such that for every \(s \in \R\), }
 s g (s) \le C \bigl(\abs{s}^2 + \abs{s}^\frac{2N}{N - 2}\bigr),\\
\tag{$g_2$}
\label{assumptStrongg}
 \lim_{s \to 0} \frac{G (s)}{\abs{s}^2} < \frac{1}{2} \qquad \text{ and } \qquad \limsup_{\abs{s} \to \infty} \frac{G (s)}{\abs{s}^\frac{2 N}{N - 2}} \le 0,\\
\label{assumptNontrivialg}
\tag{$g_3$} \text{there exists \(s_0 \in \R \setminus \{0\}\) such that \(G (s_0) > \frac{s_0^2}{2}\)},
\end{gather*}
where \(G (s) = \int_{0}^s g (\sigma)\dif\sigma\)  (and if \(g = F f\), then \(G = \frac{F^2}{2}\)).
They also proved that if \(u \in L^\infty_{\mathrm{loc}} (\R^N)\) is a finite energy solution of \eqref{problemBL} then \(u\) satisfies the Poho\v zaev identity  \cite{BerestyckiLions1983}*{Proposition 1}
\begin{equation}\label{PohozaevBL}
  \frac{N - 2}{2} \int_{\R^N} \abs{\nabla u}^2 + \frac{N}{2} \int_{\R^N} \abs{u}^2
  = N \int_{\R^N} G (u).
\end{equation}
This, in particular, implies that assumptions \eqref{assumptWeakg}, \eqref{assumptStrongg} and \eqref{assumptNontrivialg}
are ``almost necessary'' for the existence of nontrivial finite energy solutions of \eqref{problemBL}.
Indeed, the necessity of \eqref{assumptNontrivialg} follows directly from \eqref{PohozaevBL}.
For  \eqref{assumptWeakg} and \eqref{assumptStrongg}, if \(f (s) = s^p\) with \(s \not \in (1, \frac{N + 2}{N - 2})\) then \eqref{PohozaevBL} immediately implies that \eqref{problemBL} does not have any bounded finite-energy nontrivial solution.

In this spirit, we prove the existence of solutions to Choquard equation \eqref{problemBL},
assuming that nonlinearity \(f\in C(\R;\R)\) satisfies the \emph{growth assumption}\/:
\begin{equation}
 \tag{$f_1$}\label{assumptWeakf} \text{there exists \(C > 0\) such that for every \(s \in \R\), }
 \abs{s f (s)} \le C \bigl(\abs{s}^\frac{N + \alpha}{N} + \abs{s}^\frac{N + \alpha}{N - 2}\bigr),
\end{equation}
its antiderivative \(F : s \in \R \mapsto \int_{0}^s f (\sigma)\dif\sigma \) is \emph{subcritical}\/:
\begin{equation}
\tag{$f_2$}
\label{assumptStrongf}
 \lim_{s \to 0} \frac{F (s)}{\abs{s}^\frac{N + \alpha}{N}} = 0 \qquad \text{ and } \qquad \lim_{\abs{s} \to \infty} \frac{F (s)}{\abs{s}^\frac{N + \alpha}{N - 2}} = 0,
\end{equation}
and \emph{nontrivial}:
\begin{equation}
\label{assumptNontrivialf}
\tag{$f_3$} \text{there exists \(s_0 \in \R \setminus \{0\}\) such that \(F (s_0) \ne 0\)}.
\end{equation}
We call any weak solution \(u\)  of \eqref{problemMain} which lies in the Sobolev space \( H^1 (\R^N) \setminus \{0\}\) a \emph{groundstate} of \eqref{problemMain} if
\begin{equation}
\label{GroundStateC}
 \mathcal{I} (u) = c:= \inf \bigl\{ \mathcal{I} (v) \st v \in H^1 (\R^N) \setminus \{0\} \text{ is a weak solution of \eqref{problemMain}}\bigr\}.
\end{equation}

\begin{theorem}[Existence of a groundstate]
\label{theoremExistence}
Assume that \(N \ge 3\) and \(\alpha \in (0, N)\). If \(f \in C (\R; \R)\) satisfies \eqref{assumptWeakf}, \eqref{assumptStrongf} and \eqref{assumptNontrivialf}, then \eqref{problemMain} has a groundstate.
\end{theorem}

Similarly to assumptions \eqref{assumptWeakg}, \eqref{assumptStrongg} and \eqref{assumptNontrivialg} in the case of the local problem \eqref{problemMain}, our assumptions \eqref{assumptWeakf}, \eqref{assumptStrongf} and \eqref{assumptNontrivialf} are ``almost necessary'' for the existence of nontrivial solutions to \eqref{problemMain}. Indeed, by the Poho\v zaev identity (proposition~\ref{propositionPohozhaev}), any solution \(u \in H^1 (\R^N)\),
\begin{equation}
\label{eqPohozaev}
 \frac{N - 2}{2} \int_{\R^N} \abs{\nabla u}^2 + \frac{N}{2} \int_{\R^N} \abs{u}^2
 = \frac{N + \alpha}{2} \int_{\R^N} \bigl(I_\alpha \ast F (u)\bigr) F (u),
\end{equation}
problem \eqref{problemMain} does not have nontrivial solutions in \(H^1 (\R^N)\) if \(f (s) = s^p\) with \(p \not \in (\frac{\alpha}{N}, \bfrac{\alpha + 2}{N - 2})\)  (see also \cite{MorozVanSchaftingenGround}*{theorem 2}).
If \eqref{assumptNontrivialf} is not satisfied, then a solution \(u \in H^1 (\R^N)\) would satisfy \(- \Delta u + u = 0\) and then necessarily be trivial.

In the limit \(\alpha \to 0\), the assumptions~\eqref{assumptWeakf}, \eqref{assumptStrongf} and \eqref{assumptNontrivialf} do not allow to recover exactly \eqref{assumptWeakg}, \eqref{assumptStrongg} and \eqref{assumptNontrivialg}. The gap between \eqref{assumptWeakf} when \(\alpha \to 0\) and \eqref{assumptWeakg} is purely technical. When \(\alpha \to 0\), \eqref{assumptStrongf} gives the assumptions \(\lim_{s \to 0} {F (s)^2}/{\abs{s}^2} = 0\) and \(\lim_{\abs{s} \to \infty} F (s)^2 /\abs{s}^\bfrac{2 N}{N - 2} = 0\), which is stronger than \eqref{assumptStrongg}. The first assumption is not really surprising, as it can be observed that in \eqref{problemBL} both \(g (u)\) and \(u\) have the same spatial homogeneity and therefore by scaling it could always be assumed that \(\lim_{s \to 0} \bfrac{G (s)}{s^2} = 0\). The second assumption is equivalent to \(\limsup_{\abs{s} \to \infty} \bfrac{F (s)^2}{\abs{s}^\bfrac{2 N}{N - 2}} \le 0\). Finally \eqref{assumptNontrivialf} gives \(G (s) = F (s)^2 \ge 0\), which
is actually  weaker than \eqref{assumptNontrivialg}. This weakening of the condition can also be explained by the difference between the various scalings of the problem \eqref{problemMain}.

We also obtain some qualitative properties of groundstates of \eqref{problemMain}.

\begin{theorem}[Qualitative properties of groundstates]
\label{theoremSymmetry}
Assume that \(N \ge 3\), \(\alpha \in (0, N)\) and \(f \in C (\R; \R)\) satisfies \eqref{assumptWeakf}.
If \(f\) is odd and has constant sign on \((0, \infty)\) then every groundstate of \eqref{problemMain} has constant sign and
is radially symmetric with respect to some point in \(\R^N\).
\end{theorem}

Before explaining the proof of theorem~\ref{theoremExistence}, we recall the strategy of the proof of H.\thinspace Berestycki and P.-L.\thinspace Lions of the existence of solutions to \eqref{problemBL} \cite{BerestyckiLions1983}*{\S 3}. They consider the constrained minimization problem
\begin{equation}\label{minBL}
 \min \Bigl\{ \int_{\R^N} \abs{\nabla u}^2 \st u \in H^1 (\R^N)\; \text{ and } \int_{\R^N} G (u) - \frac{\abs{u}^2}{2} = 1 \Bigr\};
\end{equation}
they first show that by the P\'olya--Szeg\H o inequality for the Schwarz symmetrization, the minimum can be taken on radial and radially nonincreasing functions. Then they show the existence of  a minimum \(v \in H^1 (\R^N)\) by the direct method in the calculus of variations. This minimum \(v\) satisfies the equation
\[
  -\Delta v = \theta \bigl(g (v) - v\bigr)\quad \text{in \(\R^N\)},
\]
with a Lagrange multiplier \(\theta > 0\). They conclude by noting that \(u \in H^1 (\R^N)\) defined for \(x \in \R^N\) by \(u (x) = v (x/\sqrt{\theta})\) solves \eqref{problemBL}.

The approach of H.\thinspace Berestycki and P.-L.\thinspace Lions fails for \eqref{problemMain} for two different reasons. First, the nonlocal term will not be preserved or controlled under Schwarz symmetrization unless the nonlinearity \(f\) satisfies the more restrictive assumption of theorem~\ref{theoremSymmetry}. Second, the final scaling argument fails: the three terms in \eqref{problemMain} scale differently in space, so one cannot hope to get rid of a Lagrange multiplier by scaling in space.

In order to prove the existence of solutions in section~\ref{sectionExistence},
instead of the constrained minimization problem of type \eqref{minBL},
we consider the mountain pass level
\begin{equation}
\label{definitionb}
 b = \inf_{\gamma \in \Gamma} \sup_{t \in [0, 1]} \mathcal{I} \bigl(\gamma (t)\bigr),
\end{equation}
where the set of paths is defined as
\begin{equation}
\label{definitionGamma}
  \Gamma = \bigl\{ \gamma \in C \bigl([0, 1]; H^1 (\R^N)\bigr) \st \gamma (0) = 0 \text{ and } \mathcal{I} (\gamma (1)) < 0 \bigr\}.
\end{equation}
Classically, in order to show that \(b\) is a critical level of the functional \(\mathcal{I}\), one constructs a \emph{Palais--Smale sequence} at the level \(b\), that is, a sequence \((u_n)_{n \in \N}\) in \(H^1 (\R^N)\) such that \(\mathcal{I} (u_n) \to b\) and \(\mathcal{I}' (u_n) \to 0\) as \(n \to \infty\). Then one proves that the sequence \((u_n)_{n \in \N}\) converges up to translations and extraction of a subsequence \citelist{\cite{WillemMinimax}\cite{Struwe2008}}.
The first step of this approach is to establish the \emph{boundedness} of the sequence \((u_n)_{n \in \N}\) in \(H^1 (\R^N)\).
Usually this involves an Ambrosetti--Rabinowitz type superlinearity assumption,
which in our setting would require the existence of \(\mu > 1\)
such that \(s \in \R^+ \mapsto F (s)/s^\mu\) is nondecreasing.

In order to avoid introducing an Ambrosetti--Rabinowitz type condition, in section 2 we employ a technique introduced by L.\thinspace Jeanjean,
which consists in constructing a Palais--Smale sequence that \emph{satisfies asymptotically the Poho\v zaev identity} \cite{Jeanjean1997} (see also \cite{HirataIkomaTanaka}). This improvement is related to the monotonicity trick of M.\thinspace Struwe \cite{Struwe2008}*{\S II.9} and L. Jeanjean \cite{Jeanjean1999}.
This allows to prove the existence of a nontrivial solution \(u\) to \eqref{problemMain}
under the assumptions \eqref{assumptWeakf}, \eqref{assumptStrongf} and \eqref{assumptNontrivialf} only.
A novelty in our proof, apart of the presence of the nonlocal term in the equation,
is that we combine L.\thinspace Jeanjean method with a concentration-compactness argument.

To conclude that such constructed solution \(u\) is a groundstate, we first show \(\mathcal{I} (u) = b\).
This is  a straightforward computation if \(u\) satisfies the Poho\v zaev identity \eqref{eqPohozaev} proved in section~\ref{sectionPohozaev}. This however brings a regularity issue, as the proof of the identity \eqref{eqPohozaev} requires a little more regularity than \(u \in H^1 (\R^N)\). The subcriticality assumption \eqref{assumptWeakf} is too weak for a direct bootstrap argument. Thus we study the regularity of \(u\) in section~\ref{sectionBrezisKato} by a variant
of the Brezis--Kato regularity result \cite{BrezisKato1979}. The relationship between critical levels \(b\) and \(c\) is established with the construction of paths associated to critical points in section~\ref{sectionPaths} following L.\thinspace Jeanjean and H.\thinspace Tanaka \cite{JeanjeanTanaka2003}.

The qualitative properties of the groundstate of theorem~\ref{theoremSymmetry} are established in section~\ref{sectionSymmetry}. We show that the absolute value of a groundstate and its polarization are also groundstates.
This leads to contradiction with the strong maximum principle if the solution is not invariant under these transformations.

Finally in section~\ref{sectionSymmetricMountainPass} we explain how the proof of theorem~\ref{theoremExistence} can be simplified under the assumptions of theorem~\ref{theoremSymmetry} using symmetric mountain pass \cite{VanSchaftingen2005}, adapting the original argument of Berestycki and Lions for \eqref{problemMain}.

\section{Construction of a solution}
\label{sectionExistence}

\subsection{Construction of a Poho\v zaev--Palais--Smale sequence}
We first prove that there is a sequence of almost critical points at the level \(b\) defined in \eqref{definitionb} that satisfies asymptotically \eqref{eqPohozaev}.
We define the Poho\v zaev functional \(\mathcal{P} : H^1 (\R^N) \to \R\) for \(u \in H^1 (\R^N)\) by
\[
 \mathcal{P} (u) = \frac{N - 2}{2} \int_{\R^N} \abs{\nabla u}^2 + \frac{N}{2} \int_{\R^N} \abs{u}^2
 - \frac{N + \alpha}{2} \int_{\R^N} \bigl(I_\alpha \ast F (u)\bigr) F (u).
\]

\begin{proposition}[Construction of  a Poho\v zaev--Palais--Smale sequence]
\label{propositionMinimax}
If \(f \in C (\R; \R)\) satisfies \eqref{assumptWeakf} and \eqref{assumptNontrivialf}, then there exists a sequence \((u_n)_{n \in \N}\) in \(H^1 (\R^N)\) such that, as \(n \to \infty\),
\begin{align*}
 \mathcal{I} (u_n) &\to b > 0,\\
 \mathcal{I}' (u_n) &\to 0 \qquad\qquad \text{strongly in \( \bigl(H^{1} (\R^N)\bigr)'\)},\\
 \mathcal{P} (u_n) &\to 0.
\end{align*}
\end{proposition}

\begin{proof}
Our strategy consists in first proving in claims~\ref{claimFiniteb} and \ref{claimPositiveb} that the functional \(\mathcal{I}\) has the mountain pass geometry before concluding by a minimax principle.

\begin{claim}
\label{claimFiniteb}
The critical level satisfies
\[
 b < \infty.
\]
\end{claim}

\begin{proofclaim}
We need to show that the set of paths \(\Gamma\) is nonempty.
In view of the definition of \(\Gamma\), it is sufficient to construct \(u \in H^1 (\R^N)\) such that
\(\mathcal{I} (u) < 0\).
If we choose \(s_0\) of assumption \eqref{assumptNontrivialf} so that \(F (s_0) \ne 0\)
and set \(w = s_0 \chi_{B_1}\), we obtain
\[
 \int_{\R^N} \bigl(I_\alpha \ast F (w) \bigr) F (w)
 = F (s_0)^2 \int_{B_1} \int_{B_1} I_\alpha (x - y) > 0.
\]
By \eqref{assumptWeakf} the left hand side is continuous in \(L^2 (\R^N) \cap L^{\frac{2 N}{N  - 2}} (\R^N)\).
Since \(H^1 (\R^N)\) is dense in \(L^2 (\R^N) \cap L^{\frac{2 N}{N  - 2}} (\R^N)\),
there exists \(v \in H^1 (\R^N)\) such that
\[
 \int_{\R^N} \bigl(I_\alpha \ast F (v) \bigr) F (v) > 0.
\]
We will take the function \(u\) in the family of functions \(u_\tau \in H^1 (\R^N)\) defined for \(\tau > 0\) and \(x \in \R^N\) by
\(
  u_\tau (x) = v \bigl(\tfrac{x}{\tau}\bigr).
\)
On this family, we compute for every \(\tau > 0\),
\[
  \mathcal{I} (u_\tau) =
  \frac{\tau^{N - 2}}{2} \int_{\R^N} \abs{\nabla v}^2 + \frac{\tau^N}{2} \int_{\R^N} \abs{v}^2 - \frac{\tau^{N + \alpha}}{2} \int_{\R^N} \bigl(I_\alpha \ast F (v)\bigr) F (v),
\]
and observe that for \(\tau > 0\) large enough, \(\mathcal{I} (u_\tau) < 0\).
\end{proofclaim}

\begin{claim}
\label{claimPositiveb}
The critical level satisfies
\[
 b > 0.
\]
\end{claim}

\begin{proofclaim}
Recall the Hardy--Littlewood--Sobolev inequality \cite{LiebLoss2001}*{theorem 4.3}: if \(s \in (1, \frac{N}{\alpha})\)
then for every \(v \in L^s (\R^N)\), \(I_\alpha \ast v\in L^\frac{N s}{N - \alpha s} (\R^N)\) and
\begin{equation}
\label{eqHLS}
 \int_{\R^N} \abs{I_\alpha \ast v}^\frac{N s}{N - \alpha s} \le C \Bigl(\int_{\R^N} \abs{v}^s \Bigr)^\frac{N}{N - \alpha s},
\end{equation}
where \(C>0\) depends only on \(\alpha\), \(N\) and \(s\).
By the upper bound \eqref{assumptWeakf} on \(F\), for every \(u \in H^1 (\R^N)\),
\[
\begin{split}
 \int_{\R^N} \bigl(I_\alpha \ast F (u)\bigr) F (u)
 &\le C \Bigl( \int_{\R^N} \abs{F (u)}^{\frac{2 N}{N + \alpha}} \Bigr)^{1 + \frac{\alpha}{N}}\\
 &\le C'\Bigl( \int_{\R^N} \abs{u}^2 + \abs{u}^\frac{2 N}{N - 2} \Bigr)^{1 + \frac{\alpha}{N}}\\
 & \le C''\biggl(\Bigl( \int_{\R^N} \abs{u}^2 \Bigr)^{1 + \frac{\alpha}{N}} +
 \Bigl(\int_{\R^N} \abs{\nabla u}^2 \Bigr)^{1 + \frac{\alpha + 2}{N - 2}}\biggr).
\end{split}
\]
Hence there exists \(\delta > 0\) such that if \(\int_{\R^N} \abs{\nabla u}^2 + \abs{u}^2 \le \delta\), then
\[
 \int_{\R^N} \bigl(I_\alpha \ast F (u)\bigr) F (u)
 \le \frac{1}{4} \int_{\R^N} \abs{\nabla u}^2 + \abs{u}^2
\]
and therefore
\[
  \mathcal{I} (u) \ge \frac{1}{4} \int_{\R^N} \abs{\nabla u}^2 + \abs{u}^2.
\]
In particular, if \(\gamma \in \Gamma\), then \(\int_{\R^N} \abs{\nabla \gamma (0)}^2 + \abs{\gamma (0)}^2 = 0 <  \delta < \int_{\R^N} \abs{\nabla \gamma (1)}^2 + \abs{\gamma (1)}^2\) and by the intermediate value theorem there exists \(\Bar{\tau} \in (0, 1)\) such that \(\int_{\R^N} \abs{\nabla \gamma (\Bar{\tau})}^2 + \abs{\gamma (\Bar{\tau})}^2 = \delta\).
At this point \(\Bar{\tau}\),
\[
 \frac{\delta}{4} \le \mathcal{I} \bigl(\gamma (\Bar{\tau})\bigr) \le \sup_{\tau \in [0, 1]}\mathcal{I} (\gamma (\tau)).
\]
Since \(\gamma \in \Gamma\) is arbitrary, this implies that \(b \ge \frac{\delta}{4} > 0\).
\end{proofclaim}

\noindent\textbf{Conclusion.} Following L.\thinspace Jeanjean \cite{Jeanjean1997}*{\S 2} (see also \cite{HirataIkomaTanaka}*{\S 4}), we define the map \(\Phi : \R \times H^1 (\R^N) \to H^1 (\R^N)\) for \(\sigma \in \R\), \(v \in H^1 (\R^N)\) and \(x \in \R^N\) by
\[
 \Phi (\sigma, v) (x) = v (e^{-\sigma} x).
\]
For every \(\sigma \in \R\) and \(v \in H^1 (\R^N)\), the functional \(\mathcal{I} \circ \Phi\) is computed as
\[
  \mathcal{I} \bigl(\Phi (\sigma, v) \bigr)
  = \frac{e^{(N - 2) \sigma}}{2}
  \int_{\R^N} \abs{\nabla v}^2 + \frac{e^{N \sigma}}{2}\int_{\R^N} \abs{v}^2 - \frac{e^{(N + \alpha) \sigma}}{2} \int_{\R^N} \bigl(I_\alpha \ast F (v)\bigr) F (v).
\]
In view of \eqref{assumptWeakf}, \(\mathcal{I} \circ \Phi\) is continuously Fr\'echet--differentiable on \(\R \times H^1 (\R^N)\).
We define the family of paths
\[
 \Tilde{\Gamma} = \Bigl\{ \Tilde{\gamma} \in C \bigl([0, 1]; \R \times H^1 (\R^N)\bigr) \st \Tilde{\gamma}(0) = (0, 0) \text{ and } (\mathcal{I} \circ \Phi) \bigl(\Tilde{\gamma} (1)\bigr) < 0 \Bigr\}
\]
As \(\Gamma = \{ \Phi \circ \Tilde{\gamma} \st \Tilde{\gamma} \in \Tilde{\Gamma} \}\), the mountain pass levels of \(\mathcal{I}\) and \(\mathcal{I} \circ \Phi\) coincide:
\[
 b = \inf_{\Tilde{\gamma} \in \Tilde{\Gamma}} \sup_{\tau \in [0, 1]} (\mathcal{I} \circ \Phi) \bigl(\Tilde{\gamma} (\tau)\bigr).
\]

By the minimax principle \cite{WillemMinimax}*{theorem 2.9}, there exists a sequence \( \bigl((\sigma_n, v_n)\bigr)_{n \in \N}\) in \(\R \times H^1 (\R^N)\) such that as \(n \to \infty\)
\begin{align*}
  (\mathcal{I} \circ \Phi) (\sigma_n, v_n) &\to b,\\
  (\mathcal{I} \circ \Phi)' (\sigma_n, v_n) &\to 0 \qquad \qquad \text{in \(\bigl( \R \times H^{1} (\R^N)\bigr)^*\)}.
\end{align*}
Since for every \((h, w) \in \R \times H^1 (\R^N)\),
\[
 (\mathcal{I} \circ \Phi)' (\sigma_n, v_n)[h, w] = \mathcal{I}'\bigr(\Phi (\sigma_n, v_n)\bigr)[\Phi (\sigma_n, w)]
 + \mathcal{P} \bigl(\Phi (\sigma_n, v_n)\bigr) h,
\]
we reach the conclusion by taking \(u_n = \Phi (\sigma_n, v_n)\).
\end{proof}

\subsection{Convergence of Poho\v zaev--Palais--Smale sequences}
We will now show how a solution of problem \eqref{problemMain} can be constructed from the sequence given by proposition~\ref{propositionMinimax}.

\begin{proposition}[Convergence of Poho\v zaev--Palais--Smale sequences]
\label{propositionPPS}
Let \(f \in C (\R; \R)\) and  \((u_n)_{n \in \N}\) be a sequence in \(H^1 (\R^N)\).
If \(f\) satisfies \eqref{assumptWeakf} and \eqref{assumptStrongf}, \(\bigl(\mathcal{I} (u_n)\bigr)_{n \in \N}\) is bounded  and, as \(n \to \infty\),
\begin{align*}
 \mathcal{I}' (u_n) &\to 0 \qquad\qquad \text{strongly in \( (H^{1} (\R^N))'\)}, \\
 \mathcal{P} (u_n) & \to 0,
\end{align*}
then
\begin{itemize}[--]
 \item either up to a subsequence \(u_n \to 0\) strongly in \(H^1 (\R^N)\),
 \item or there exists \(u \in H^1 (\R^N) \setminus \{0\}\) such that \(\mathcal{I}' (u) = 0\) and a sequence \((a_n)_{n \in \N}\) of points in \(\R^N\)  such that up to a subsequence \(u_n (\cdot - a_n) \weakto u\) weakly in \(H^1 (\R^N)\) as \(n \to \infty\).
\end{itemize}
\end{proposition}

\begin{proof} \setcounter{claim}{0}
Assume that the first part of the alternative does not hold, that is,
\begin{equation}\label{eqNotun0}
  \liminf_{n \to \infty} \int_{\R^n} \abs{\nabla u_n}^2 + \abs{u_n}^2 > 0.
\end{equation}
We first establish in claim~\ref{claimBounded} the boundedness of the sequence and then the nonvanishing of the sequence in claim~\ref{claimNonVanishing}.

\begin{claim}
\label{claimBounded}
The sequence \((u_n)_{n \in \N}\) is bounded in \(H^1 (\R^N)\).
\end{claim}
\begin{proofclaim}[Proof of claim~\ref{claimBounded}]
For every \(n \in \N\),
\begin{equation*}
 \frac{\alpha + 2}{2(N + \alpha)}
 \int_{\R^N} \abs{\nabla u_n}^2 + \frac{\alpha}{2 (N + \alpha)} \int_{\R^N} \abs{u_n}^2
=
 \mathcal{I} (u_n) - \frac{1}{N + \alpha} \mathcal{P} (u_n).
\end{equation*}
As the right-hand side is bounded by our assumptions, the sequence \((u_n)_{n \in \N}\) is bounded in \(H^1 (\R^N)\).
\end{proofclaim}

\begin{claim}
\label{claimNonVanishing}
For every \(p \in (2, \frac{2 N}{N - 2})\),
\[
 \liminf_{n \to \infty} \sup_{a \in \R^N} \int_{B_1 (a)} \abs{u_n}^p > 0.
\]
\end{claim}

\begin{proofclaim}
First, by \eqref{eqNotun0} and the definition of the Poho\v zaev functional \(\mathcal{P}\) we have
\begin{equation}
\label{eqLiminfNonlinear}
\begin{split}
\liminf_{n \to \infty}  \int_{\R^N} \bigl(I_\alpha \ast F (u_n)\bigr) F (u_n) &= \liminf_{n \to \infty} \frac{N - 2}{N + \alpha} \int_{\R^N} \abs{\nabla u}^2 + \frac{N}{N + \alpha} \int_{\R^N} \abs{u}^2
 - \frac{2}{N + \alpha} \mathcal{P} (u_n)\\
 &> 0.
\end{split}
\end{equation}
For every \(n \in \N\), the function \(u_n\) satisfies the inequality
\citelist{\cite{Lions1984CC2}*{lemma I.1}\cite{WillemMinimax}*{lemma 1.21}\cite{MorozVanSchaftingenGround}*{lemma 2.3}}
\[
 \int_{\R^N} \abs{u_n}^p \le C \Bigl(\int_{\R^N} \abs{\nabla u_n}^2 + \abs{u_n}^2\Bigr)
 \Bigl(\sup_{a \in \R^N} \int_{B_1 (a)} \abs{u_n}^p\Bigr)^{1 - \frac{2}{p}}.
\]
As \(F\) is continuous and satisfies \eqref{assumptStrongf}, for every \(\epsilon > 0\), there exists \(C_\epsilon\) such that for every \(s \in \R\),
\[
 \abs{F (s)}^\frac{2 N}{N + \alpha} \le \epsilon \bigl(\abs{s}^2 + \abs{s}^\frac{2 N}{N - 2}\bigr) + C_\epsilon \abs{s}^p.
\]
Since \((u_n)_{n \in \N}\) is bounded in \(H^1 (\R^N)\) and hence, by the Sobolev embedding,
in \(L^{\frac{2 N}{N - 2}} (\R^N)\), we have
\[
 \liminf_{n \to \infty} \int_{\R^N} \abs{F (u_n)}^\frac{2 N}{N + \alpha}
 \le C'' \epsilon
 + C_\epsilon'
 \Bigl(\liminf_{n\to\infty} \sup_{a \in \R^N} \int_{B_1 (a)} \abs{u_n}^p\Bigr)^{1 - \frac{2}{p}}.
\]
Since \(\epsilon > 0\) is arbitrary, if \(\liminf_{n \to \infty} \sup_{a \in \R^N} \int_{B_1 (a)} \abs{u_n}^p = 0\),
then
\[
 \liminf_{n \to \infty} \int_{\R^N} \abs{F (u_n)}^\frac{2 N}{N + \alpha}
  = 0,
\]
and the Hardy--Littlewood--Sobolev inequality implies that
\[
 \liminf_{n \to \infty} \int_{\R^N} \bigl(I_\alpha \ast F (u_n)\bigr) F (u_n) = 0,
\]
in contradiction with \eqref{eqLiminfNonlinear}.
\end{proofclaim}

\noindent\textbf{Conclusion.}
Up to a translation, we can now assume that for some \(p \in (2, \frac{2 N}{N - 2})\), \[\liminf_{n \to \infty} \int_{B_1} \abs{u_n}^p > 0.\] By Rellich's theorem, this implies that up to a subsequence, \((u_n)_{n \in \N}\) converges weakly in \(H^1 (\R^N)\) to \(u \in H^1 (\R^N) \setminus \{0\}\).

As the sequence \((u_n)_{n \in \N}\) is bounded in \(H^1 (\R^N)\), by the Sobolev embedding, it is also bounded  in \(L^2 (\R^N) \cap L^\frac{2 N}{N - 2} (\R^N)\). By \eqref{assumptWeakf}, the sequence \((F \circ u_n)_{n \in \N}\) is therefore bounded in \(L^{\frac{2 N}{N + \alpha}} (\R^N)\).
Since the sequence \((u_n)_{n \in \N}\) converges weakly to \(u\) in \(H^1 (\R^N)\),
it converges up to a subsequence  to \(u\) almost everywhere in \(\R^{N}\).
By continuity of \(F\), \((F \circ u_n)_{n \in \N}\) converges almost everywhere to \(F \circ u\) in \(\R^N\). This implies that the sequence \((F \circ u_n)_{n \in \N}\) converges weakly to \(F \circ u\) in \(L^\frac{2 N}{N + \alpha} (\R^N)\). As the Riesz potential defines a linear continuous map from \(L^\frac{2 N}{N + \alpha} (\R^N)\) to \(L^\frac{2 N}{N - \alpha} (\R^N)\), the sequence \((I_\alpha \ast (F \circ u_n))_{n \in \N}\) converges weakly to \(I_\alpha \ast (F \circ u)\) in \(L^\frac{2 N}{N - \alpha} (\R^N)\).

On the other hand, in view of \eqref{assumptWeakf} and by Rellich's theorem, the sequence \((f \circ u_n)_{n \in \N}\) converges strongly to \(f \circ u\) in \(L^p_{\mathrm{loc}} (\R^N)\) for every \(p \in [1, \frac{2 N}{\alpha + 2})\). We conclude that
\[
  \big(I_\alpha \ast (F \circ u_n)\bigr) (f \circ u_n)  \weakto \bigl(I_\alpha \ast (F \circ u)\bigr) (f \circ u) \qquad \text{weakly in \(L^p (\R^N)\)},
\]
for every \(p \in [1, \frac{2 N}{N + 2})\). This implies in particular that for every \(\varphi \in C^1_c (\R^N)\),
\begin{multline*}
 \int_{\R^N} \nabla u \cdot \nabla \varphi + u \varphi - \int_{\R^N} \bigl(I_\alpha \ast (F \circ u)\bigr) \bigl((f \circ u) \varphi\bigr)\\
 = \lim_{n \to \infty} \int_{\R^N} \nabla u \cdot \nabla \varphi + u \varphi - \int_{\R^N} \bigl(I_\alpha \ast (F \circ u_n)\bigr) (f \circ u_n) \varphi
 = 0;
\end{multline*}
that is, \(u\) is a weak solution of \eqref{problemMain}.
\end{proof}

We point out that the assumption~\eqref{assumptStrongf} is only used in the proof of claim~\ref{claimNonVanishing}.

\section{Regularity of solutions and Poho\v zaev identity.}

We prove in this section that any solution of \eqref{problemMain} has some additional regularity; this regularity will be sufficient to establish the Poho\v zaev identity \eqref{eqPohozaev}.

\begin{proposition}[Improved local regularity of solutions of \eqref{problemMain}]
\label{propositionRegularity}
If \(f \in C (\R; \R)\) satisfies \eqref{assumptWeakf} and \(u\in H^1(\R^N)\) solves \eqref{problemMain},
then for every \(p \ge 1\), \(u \in W^{2, p}_\mathrm{loc} (\R^N)\).
\end{proposition}

In particular, proposition~\ref{propositionRegularity} with the Morrey--Sobolev embeddings imply that a solution \(u\) is locally H\"older continuous.
If \(f\) has additional regularity then regularity of \(u\) could be further improved via Schauder estimates.

The assumption \eqref{assumptWeakf} is too weak for the standard bootstrap method as in \citelist{\cite{CingolaniClappSecchi2011}*{lemma A.1}\cite{MorozVanSchaftingenGround}*{proposition 4.1}}.
Instead, in order to prove regularity of solutions of \eqref{problemMain} we shall rely on a nonlocal version of the Brezis--Kato estimate.

\subsection{A nonlocal Brezis--Kato type estimate}
\label{sectionBrezisKato}
A special case of the regularity result of Brezis and Kato \cite{BrezisKato1979}*{theorem 2.3} states that if \(u \in H^1 (\R^N)\) is a solution
of the linear elliptic equation
\begin{equation}
\label{problemBrezisKato}
  - \Delta u + u = V u\quad \text{in \(\R^N\)},
\end{equation}
and \(V \in L^\infty (\R^N) + L^\frac{N}{2} (\R^N)\), then \(u \in L^p (\R^N)\) for every \(p \ge 1\).
We extend this result to a class of nonlocal linear equations.

\begin{proposition}[Improved integrability of solution of a nonlocal critical linear equation]
\label{propositionBrezisKato}
If
\(H, K \in L^{\frac{2N}{\alpha}} (\R^N) + L^{\frac{2 N}{\alpha + 2}} (\R^N)\) and \(u \in H^1 (\R^N)\) solves
\[
 - \Delta u + u = (I_\alpha \ast H u) K,
\]
then \(u \in L^p (\R^N)\) for every \(p \in [2, \frac{N}{\alpha} \frac{2 N}{N - 2})\).
Moreover, there exists a constant \(C_p\) independent of \(u\) such that
\[
  \Bigl(\int_{\R^N} \abs{u}^p \Bigr)^\frac{1}{p} \le C_p \Bigl(\int_{\R^N} \abs{u}^2\Bigr)^\frac{1}{2}.
\]
\end{proposition}

Our proof of proposition~\ref{propositionBrezisKato} follows the strategy of Brezis and Kato (see also Trudinger \cite{Trudinger1968}*{Theorem 3}). The adaptation of the argument is complicated by the nonlocal effect of \(u\) 
on the right hand side.

Our main new tool for the proof of proposition~\ref{propositionBrezisKato} is the following lemma, which is a nonlocal counterpart of the estimate \cite{BrezisKato1979}*{lemma 2.1}: if \(V \in L^\infty (\R^N) + L^\frac{N}{2} (\R^N)\), then for every \(\epsilon > 0\), there exists \(C_\epsilon\) such that
\begin{equation}
\label{ineqLocalElliptic}
  \int_{\R^N} V \abs{u}^2 \le \epsilon^2 \int_{\R^N} \abs{\nabla u}^2 + C_\epsilon \int_{\R^N} \abs{u}^2.
\end{equation}

\begin{lemma}
\label{lemmathetaElliptic}
Let \(N \ge 2\), \(\alpha \in (0, 2)\) and \(\theta \in (0, 2)\). If \(H, K \in L^\frac{2 N}{\alpha + 2}(\R^N) + L^\frac{2 N}{\alpha} (\R^N)\) and \(\frac{\alpha}{N} < \theta < 2 - \frac{\alpha}{N}\), then for every \(\epsilon > 0\), there exists \(C_{\epsilon, \theta} \in \R\) such that for every \(u \in H^1 (\R^N)\),
\[
 \int_{\R^N} \bigl( I_{\alpha} \ast (H \abs{u}^{\theta}) \bigr) K \abs{u}^{2 - \theta}
 \le \epsilon^2 \int_{\R^N} \abs{\nabla u}^2 + C_{\epsilon, \theta} \int_{\R^N} \abs{u}^2.
\]
\end{lemma}

In the limit \(\alpha = 0\), this result is consistent with \eqref{ineqLocalElliptic}; the parameter \(\theta\) only plays a role in the nonlocal case.

In order to prove lemma~\ref{lemmathetaElliptic}, we shall use several times the following inequality.

\begin{lemma}
\label{lemmathetaHolder}
Let \(q, r, s, t \in [1, \infty)\) and \(\lambda \in [0, 2] \) such that
\[
 1 + \frac{\alpha}{N} - \frac{1}{s} - \frac{1}{t} =  \frac{\lambda}{q} + \frac{2 - \lambda}{r},
\]
If \(\theta \in (0, 2)\) satisfies
\begin{gather*}
\min(q, r) \Bigl(\frac{\alpha}{N} - \frac{1}{s}\Bigr) < \theta < \max(q, r) \Bigl(1 - \frac{1}{s}\Bigr),\\
\min(q, r) \Bigl(\frac{\alpha}{N} - \frac{1}{t}\Bigr) < 2 - \theta < \max(q, r) \Bigl(1 - \frac{1}{t}\Bigr),
\end{gather*}
then for every  \(H \in L^s (\R^N)\), \(K \in L^t (\R^N)\) and \(u \in L^q (\R^N) \cap L^r (\R^N)\),
\[
 \int_{\R^N} (I_\alpha \ast \bigl(H \abs{u}^\theta)\bigr) K \abs{u}^{2 - \theta}
 \le C \Bigl(\int_{\R^N} \abs{H}^s \Bigr)^\frac{1}{s}\Bigl(\int_{\R^N} \abs{K}^t \Bigr)^\frac{1}{t} \Bigl(\int_{\R^N} \abs{u}^q \Bigr)^\frac{\lambda}{q}
 \Bigl(\int_{\R^N} \abs{u}^r \Bigr)^\frac{2 - \lambda}{r}.
\]
\end{lemma}
\begin{proof}
First observe that if \(\Tilde{s} > 1\), \(\Tilde{t} > 1\), satisfy \(\frac{1}{\Tilde{t}} + \frac{1}{\Tilde{s}} = 1 + \frac{\alpha}{N} \), the Hardy--Littlewood--Sobolev inequality is applicable and
\[
  \int_{\R^N} \bigl( I_{\alpha} \ast (H \abs{u}^{\theta}) \bigr) K \abs{u}^{2 - \theta}
  \le C \Bigl(\int_{\R^N} \bigabs{H u^\theta}^{\Tilde{s}} \Bigr)^{\frac{1}{\Tilde{s}}}\Bigl(\int_{\R^N} \bigabs{K u^{2 - \theta}}^{\Tilde{t}} \Bigr)^{\frac{1}{\Tilde{t}}}.
\]
Let \(\mu \in \R\). Note that if
\begin{align}
\label{eqInterpCond1}
0 &\le \mu \le \theta &
& \text{ and }&
\frac{1}{\Tilde{s}} := \frac{\mu}{q} + \frac{\theta - \mu}{r} + \frac{1}{s} &< 1,
\end{align}
then by H\"older's inequality
\[
  \Bigl(\int_{\R^N} \bigabs{H u^\theta}^{\Tilde{s}} \Bigr)^{\frac{1}{\Tilde{s}}}
  \le  \Bigl(\int_{\R^N} \abs{H}^s \Bigr)^\frac{1}{s}
\Bigl(\int_{\R^N} \abs{u}^q \Bigr)^\frac{\mu}{q}
 \Bigl(\int_{\R^N} \abs{u}^r \Bigr)^\frac{\theta - \mu}{r}.
\]
Similarly, if
\begin{align}
\label{eqInterpCond2}
\lambda - (2 - \theta) &\le \mu \le \lambda &
& \text{ and }&
\frac{1}{\Tilde{t}} := \frac{\lambda - \mu}{q} + \frac{(2 - \theta) - (\lambda-\mu)}{r} + \frac{1}{t} &< 1,
\end{align}
then
\[
  \Bigl(\int_{\R^N} \bigabs{K u^{2 - \theta}}^{\Tilde{t}} \Bigr)^{\frac{1}{\Tilde{t}}}
  \le  \Bigl(\int_{\R^N} \abs{K}^t \Bigr)^\frac{1}{t}
\Bigl(\int_{\R^N} \abs{u}^q \Bigr)^\frac{\lambda - \mu}{q}
 \Bigl(\int_{\R^N} \abs{u}^r \Bigr)^\frac{2 - \theta - (\lambda - \mu)}{r}.
\]
It can be checked that \eqref{eqInterpCond1} and \eqref{eqInterpCond2} can be satisfied for some \(\mu \in \R\) if and only the assumptions of the lemma hold.
In particular,
\(\frac{1}{\Tilde{t}} + \frac{1}{\Tilde{s}} = \frac{1}{s} + \frac{1}{t} =  \frac{\lambda}{q} + \frac{2 - \lambda}{r} = 1 + \frac{\alpha}{N} \), so that we can conclude.
\end{proof}

\begin{proof}[Proof of lemma~\ref{lemmathetaElliptic}]
Let  \(H = H^* + H_*\) and \(K = K^* + K_*\) with \(H^*, K^* \in L^\frac{2 N}{\alpha}(\R^N)\) and \(H_*, K_* \in L^\frac{2 N}{\alpha + 2} (\R^N)\).
Applying lemma~\ref{lemmathetaHolder}, with \(q = r = \frac{2 N}{N - 2}\), \(s = t = \frac{2 N}{\alpha + 2}\) and \(\lambda = 0\), we have since \(\abs{\theta - 1} < \frac{N - \alpha}{N - 2}\),
\[
 \int_{\R^N} \bigl( I_{\alpha} \ast (H_* \abs{u}^{\theta}) \bigr) (K_* \abs{u}^{2 - \theta}) \le C \Bigl(\int_{\R^N} \abs{H_*}^\frac{2 N}{\alpha + 2} \Bigr)^\frac{\alpha + 2}{2N}\Bigl(\int_{\R^N} \abs{K_*}^\frac{2 N}{\alpha + 2} \Bigr)^\frac{\alpha + 2}{2N}
\Bigl(\int_{\R^N} \abs{u}^\frac{2 N}{N - 2} \Bigr)^{1 - \frac{2}{N}}.
\]
Taking now, \(s = t = \frac{2 N}{\alpha}\), \(q= r = 2\) and \(\lambda = 2\), we have since \(\abs{\theta - 1} < \frac{N - \alpha}{N}\),
\[
 \int_{\R^N} \bigl( I_{\alpha} \ast (H^* \abs{u}^{\theta}) \bigr) (K^* \abs{u}^{2 - \theta}) \le C \Bigl(\int_{\R^N} \abs{H^*}^\frac{2 N}{\alpha} \Bigr)^\frac{\alpha}{2N}
 \Bigl(\int_{\R^N} \abs{K^*}^\frac{2 N}{\alpha} \Bigr)^\frac{\alpha}{2N}
\int_{\R^N} \abs{u}^2.
\]
Similarly, with
\(s = \frac{2 N}{\alpha + 2}\), \(t = \frac{2 N}{\alpha}\), \(q=2\), \(r=\frac{2 N}{N - 2}\)  and \(\lambda = 1\),
\begin{multline*}
 \int_{\R^N} \bigl( I_{\alpha} \ast (H_* \abs{u}^{\theta}) \bigr) (K^* \abs{u}^{2 - \theta}) \\
 \le C \Bigl(\int_{\R^N} \abs{H_*}^\frac{2 N}{\alpha + 2} \Bigr)^\frac{\alpha + 2}{2N}
\Bigl(\int_{\R^N} \abs{K^*}^\frac{2 N}{\alpha} \Bigr)^\frac{\alpha}{2N}
\Bigl(\int_{\R^N} \abs{u}^2\Bigr)^\frac{1}{2}\Bigl(\int_{\R^N} \abs{u}^\frac{2 N}{N - 2} \Bigr)^{\frac{1}{2} - \frac{1}{N}}
\end{multline*}
and with \(s = \frac{2 N}{\alpha}\), \(t = \frac{2 N}{\alpha + 2}\),  \(q=2\), \(r=\frac{2 N}{N - 2}\) and \(\lambda = 1\),
\begin{multline*}
 \int_{\R^N} \bigl( I_{\alpha} \ast (H^* \abs{u}^{\theta}) \bigr) (K_* \abs{u}^{2 - \theta}) \\
 \le C \Bigl(\int_{\R^N} \abs{H^*}^\frac{2 N}{\alpha} \Bigr)^\frac{\alpha}{2N} \Bigl(\int_{\R^N} \abs{K_*}^\frac{2 N}{\alpha + 2} \Bigr)^\frac{\alpha + 2}{2N}
\Bigl(\int_{\R^N} \abs{u}^2\Bigr)^\frac{1}{2} \Bigl(\int_{\R^N} \abs{u}^\frac{2 N}{N - 2} \Bigr)^{\frac{1}{2} - \frac{1}{N}}.
\end{multline*}
By the Sobolev inequality, we have thus proved that for every \(u \in H^1 (\R^N)\),
\begin{multline*}
 \int_{\R^N} \bigl( I_{\alpha} \ast (H \abs{u}^{\theta}) \bigr) (K \abs{u}^{2 - \theta})\\
 \le C \biggl(\Bigl(\int_{\R^N} \abs{H_*}^\frac{2 N}{\alpha + 2} \int_{\R^N} \abs{K_*}^\frac{2 N}{\alpha + 2} \Bigr)^\frac{\alpha + 2}{2N}
\int_{\R^N} \abs{\nabla u}^2\\
+ \Bigl(\int_{\R^N} \abs{H^*}^\frac{2 N}{\alpha} \int_{\R^N} \abs{K^*}^\frac{2 N}{\alpha} \Bigr)^\frac{\alpha}{2 N}
\int_{\R^N} \abs{u}^2  \biggr).
\end{multline*}
The conclusion follows by choosing \(H^*\) and \(K^*\) such that
\[
C \Bigl(\int_{\R^N} \abs{H_*}^\frac{2 N}{\alpha + 2} \int_{\R^N} \abs{K_*}^\frac{2 N}{\alpha + 2} \Bigr)^\frac{\alpha + 2}{2N} \le \epsilon^2.\qedhere
\]
\end{proof}

\begin{proof}[Proof of proposition~\ref{propositionBrezisKato}]
By Lemma~\ref{lemmathetaElliptic} with \(\theta = 1\), there exists \(\lambda > 0\) such that for every \(\varphi \in H^1 (\R^N)\),
\[
 \int_{\R^n} \bigl(I_\alpha \ast \abs{H \varphi}\bigr) \abs{K \varphi}
 \le \frac{1}{2} \int_{\R^n} \abs{\nabla \varphi}^2 + \frac{\lambda}{2} \int_{\R^N} \abs{\varphi}^2.
\]
Choose sequences \((H_k)_{k \in \N}\) and \((K_k)_{k \in \N}\) in  \(L^\frac{2 N}{\alpha} (\R^N)\) such that
\(\abs{H_k} \le \abs{H}\) and \(\abs{K_k} \le \abs{K}\), and \(H_k \to H\) and \(K_k \to K\) almost everywhere in \(\R^N\).
For each \(k \in \N\), the form \(a_k : H^1 (\R^N) \times H^1 (\R^N) \to \R\) defined for \(u \in H^1 (\R^N)\) and \(v \in H^1 (\R^N)\) by
\[
  a_k (u, v) = \int_{\R^N} \nabla u \cdot \nabla v + \lambda u v - \int_{\R^N} (I_\alpha \ast H_k u) K_k v
\]
is bilinear and coercive; by the Lax--Milgram theorem \cite{Brezis}*{corollary 5.8}, there exists a unique solution \(u_k \in H^1 (\R^N)\) of
\begin{equation}
\label{eqBrezisKatoApprox}
  - \Delta u_k + \lambda u_k = \bigl(I_\alpha \ast (H_k u_k)\bigr) K_k + (\lambda - 1) u.
\end{equation}
It can be proved that the sequence \((u_k)_{k \in \N}\) converges weakly to \(u\) in \(H^1 (\R^N)\) as \(k \to \infty\).

For \(\mu > 0\), we define the truncation \(u_{k, \mu} : \R^N \to \R\) for \(x \in \R^N\) by
\[
  u_{k, \mu} (x)
  = \begin{cases}
       -\mu & \text{if \(u_k (x) \le - \mu\)},\\
       u_k (x) & \text{if \(- \mu < u_k (x) < \mu\)},\\
       \mu & \text{if \(u_k (x) \ge \mu\)},
    \end{cases}
\]
Since \( \abs{u_{k, \mu}}^{p - 2} u_{k, \mu} \in H^1 (\R^N)\), we can take it as a test function in \eqref{eqBrezisKatoApprox}:
\begin{multline*}
 \int_{\R^N} \tfrac{4(p - 1)}{p^2} \bigabs{\nabla (u_{k, \mu})^\frac{p}{2}}^2
 + \bigabs{\abs{u_{k, \mu}}^\frac{p}{2}}^2
 \le \int_{\R^N} (p - 1) \abs{u_{k, \mu}}^{p - 2}  \bigabs{\nabla u_{k, \mu}}^2
 + \abs{u_{k, \mu}}^{p - 2} u_{k, \mu} u_k\\
 = \int_{\R^N} \bigl(I_\alpha \ast (H_k u_k)\bigr) \bigl(K_k \abs{u_{k, \mu}}^{p - 2} u_{k, \mu}\bigr) + (\lambda - 1) u\abs{u_{k, \mu}}^{p - 2} u_{k, \mu}.
\end{multline*}
If \(p < \frac{2 N}{\alpha}\), by lemma~\ref{lemmathetaElliptic} with \(\theta = \frac{2}{p}\), there exists \(C > 0\) such that
\[
\begin{split}
 \int_{\R^N} \bigl(I_\alpha \ast \abs{H_k u_{k, \mu}}\bigr) \bigl(\abs{K_k} \abs{u_{k, \mu}}^{p - 2} u_{k, \mu}\bigr)
 &\le \int_{\R^N} \bigl(I_\alpha \ast (\abs{H}\abs{u_{k, \mu}})\bigr) \bigl(\abs{K} \abs{u_{k, \mu}}^{p - 1} \bigr)\\
 &\le \tfrac{2 (p - 1)}{p^2} \int_{\R^N} \bigabs{\nabla (u_{k, \mu})^\frac{p}{2}}^2
 + C \int_{\R^N} \bigabs{\abs{u_{k, \mu}}^\frac{p}{2}}^2.
\end{split}
\]
We have thus
\[
 \tfrac{2 (p - 1)}{p^2} \int_{\R^N} \bigabs{\nabla (u_{k, \mu})^\frac{p}{2}}^2
 \le C' \int_{\R^N} \bigl(\abs{u_{k}}^p + \abs{u}^p\bigr)
 + \int_{A_{k, \mu}} \bigl(I_\alpha \ast (\abs{K_k} \abs{u_{k}}^{p - 1})\bigr) \abs{H_k u_k},
\]
where
\[
  A_{k, \mu} = \bigl\{x \in \R^N : \abs{u_{k} (x)} > \mu \bigr\}.
\]
Since \(p < \frac{2 N}{\alpha}\), by the Hardy--Littlewood--Sobolev inequality,
\[
 \int_{A_{k, \mu}} \bigl(I_\alpha \ast (\abs{K_k} \abs{u_{k}}^{p - 1})\bigr) \abs{H_k u_k}
 \le C \Bigl(\int_{\R^N} \bigabs{\abs{K_k} \abs{u_k}^{p - 1}}^r\Bigr)^\frac{1}{r}\Bigl(\int_{A_{k, \mu}} \abs{H_k u_k}^s\Bigr)^\frac{1}{s},
\]
with \(\frac{1}{r} = \frac{\alpha}{2 N} + 1 - \frac{1}{p}\) and \(\frac{1}{s} = \frac{\alpha}{2 N} +  \frac{1}{p}\). By H\"older's inequality, if \(u_k \in L^p (\R^N)\), then \(\abs{K_k} \abs{u_{k}}^{p - 1} \in L^r (\R^N)\) and \(\abs{H_k u_{k}} \in L^s (\R^N)\), whence by Lebesgue's dominated convergence theorem
\[
 \lim_{\mu \to \infty} \int_{A_{k, \mu}} \bigl(I_\alpha \ast (\abs{K_k} \abs{u_{k}}^{p - 1})\bigr) \abs{H_k u_k} = 0.
\]
In view of the Sobolev estimate, we have proved the inequality
\[
  \limsup_{k \to \infty} \Bigl(\int_{\R^N} \abs{u_k}^\frac{p N}{N - 2} \Bigr)^{1 - \frac{2}{N}}
  \le C'' \limsup_{k \to \infty} \int_{\R^N} \abs{u_k}^p.
\]
By iterating over \(p\) a finite number of times we cover the range \(p \in [2, \frac{2N}{\alpha})\).
\end{proof}

\begin{remark}
\label{remarkUniformRegularity}
A close inspection of the proofs of lemma~\ref{lemmathetaElliptic} and of proposition~\ref{propositionBrezisKato} gives a more precise dependence of the constant \(C_p\). Given a function \(M : (0, \infty) \to (0, \infty)\) and \(p \in (2, \frac{N}{\alpha} \frac{2 N}{N - 2})\), then there exists \(C_{p, M}\) such that
if for every \(\epsilon > 0\), \(K\) and \(H\) can be decomposed as \(K = K_* + K^*\) and \(H = H_* + H^*\) with
\begin{align*}
  \Bigl(\int_{\R^N} \abs{K_*}^\frac{2 N}{\alpha + 2}\Bigr)^\frac{2 N}{\alpha + 2} & \le \epsilon &
  &\text{ and }&
  \Bigl(\int_{\R^N} \abs{K^*}^\frac{2 N}{\alpha}\Bigr)^\frac{2 N}{\alpha} & \le M (\epsilon),\\
  \Bigl(\int_{\R^N} \abs{H_*}^\frac{2 N}{\alpha + 2}\Bigr)^\frac{2 N}{\alpha + 2} & \le \epsilon &
  &\text{ and }&
  \Bigl(\int_{\R^N} \abs{H^*}^\frac{2 N}{\alpha}\Bigr)^\frac{2 N}{\alpha} & \le M (\epsilon),
\end{align*}
and if \(u \in H^1 (\R^N)\) satisfies
\[
 - \Delta u + u = (I_\alpha \ast H u) K,
\]
then one has
\[
  \Bigl(\int_{\R^N} \abs{u}^p \Bigr)^\frac{1}{p} \le C_{p, M} \Bigl(\int_{\R^N} \abs{u}^2\Bigr)^\frac{1}{2}.
\]
\end{remark}

\subsection{Regularity of solutions}
Here we prove of the regularity result for the nonlinear problem~\eqref{problemMain}.

\begin{proof}[Proof of proposition~\ref{propositionRegularity}]
Define \(H : \R^N \to \R\) and \(K : \R^N \to \R\) for \(x \in \R^N\) by \(H (x) = F \bigl(u (x)\bigr)/u (x)\) and \(K (x) = f \bigl(u (x)\bigr)\). Observe that for every \(x \in \R^N\),
\[
 K (x) \le C \bigl(\abs{u (x)}^\frac{\alpha}{N} + \abs{u (x)}^\frac{\alpha+2}{N - 2}\bigr)
\]
and
\[
 H (x) \le C \bigl(\tfrac{N}{N+\alpha}\abs{u (x)}^\frac{\alpha}{N} + \tfrac{N - 2}{N + \alpha} \abs{u (x)}^\frac{\alpha+2}{N - 2}\bigr),
\]
so that \(K, H \in L^\frac{2 N}{\alpha} (\R^N) + L^\frac{2 N}{\alpha + 2} (\R^N)\). By proposition~\ref{propositionBrezisKato}, \(u \in L^p (\R^N)\) for every \(p \in [2, \frac{N}{\alpha} \frac{2 N}{N - 2})\).
In view of \eqref{assumptWeakf}, \(F \circ u \in L^q (\R^N)\) for every \(q \in [\frac{2 N}{N + \alpha}, \frac{N}{\alpha} \frac{2 N}{N + \alpha})\).
Since \(\frac{2 N}{N + \alpha} < \frac{N}{\alpha} < \frac{N}{\alpha} \frac{2 N}{N + \alpha}\), we have \(I_\alpha \ast (F \circ u) \in L^\infty (\R^N)\), and thus
\[
  \abs{- \Delta u + u} \le C \bigl(\abs{u}^\frac{\alpha}{N} + \abs{u}^\frac{\alpha + 2}{N - 2}\bigr).
\]
By the classical bootstrap method for subcritical local problems in bounded domains, we deduce that
\(u \in W^{2, p}_{\mathrm{loc}} (\R^N)\) for every \(p \ge 1\).
\end{proof}

\subsection{Poho\v zaev identity}
\label{sectionPohozaev}
The regularity information that has been gained in the previous sections allows to prove a Poho\v zaev integral identity.

\begin{proposition}[Poho\v zaev identity]
\label{propositionPohozhaev}
If \(f \in C (\R; \R)\) satisfies \eqref{assumptWeakf} and \(u\in H^1(\R^N)\) solves \eqref{problemMain},
then
\[
 \frac{N - 2}{2} \int_{\R^N} \abs{\nabla u}^2 + \frac{N}{2} \int_{\R^N} \abs{u}^2
 = \frac{N + \alpha}{2} \int_{\R^N} \bigl(I_\alpha \ast F (u)\bigr) F (u).
\]
and
\[
 \mathcal{I} (u) = \frac{\alpha + 2}{2 (N  + \alpha)} \int_{\R^N} \abs{\nabla u}^2
 + \frac{\alpha}{2 (N + \alpha)} \int_{\R^N} \abs{u}^2.
\]
\end{proposition}

This proposition implies in particular that if \(u \ne 0\), then
\[
 \mathcal{I} (u) > 0.
\]

The proof of this proposition is a generalization of the argument for \(f (s) = s^p\) \cite{MorozVanSchaftingenGround} (see also particular cases \citelist{\cite{Menzala-1983}\cite{CingolaniSecchiSquassina2010}*{lemma 2.1}}).
The strategy is classical and consists in testing the equation against a suitable cut-off of \(x \cdot \nabla u(x)\) and integrating by parts \citelist{\cite{Kavian1993}*{proposition 6.2.1}\cite{WillemMinimax}*{appendix B}}.

\begin{proof}[Proof of proposition~\ref{propositionPohozhaev}]
By proposition~\ref{propositionRegularity}, \(u \in W^{2, 2}_\mathrm{loc} (\R^N)\).
Fix \(\varphi \in C^1_c(\R^N)\) such that \(\varphi = 1\) in a neighbourhood of \(0\).
The function \(v_\lambda \in W^{1, 2} (\R^N)\) defined for \(\lambda \in (0, \infty)\) and \(x \in \R^N\) by
\[
 v_\lambda (x) = \varphi(\lambda x)\, x \cdot \nabla u (x),
\]
can be used as a test function in the equation to obtain
\[
   \int_{\R^N} \nabla u \cdot \nabla v_\lambda + \int_{\R^N} u v_\lambda
= \int_{\R^N} \bigl(I_\alpha \ast F (u)\bigr) (f (u) v_\lambda).
\]
The left-hand side can be computed by integration by parts for every \(\lambda > 0\) as
\[
\begin{split}
 \int_{\R^N} u v_\lambda &= \int_{\R^N} u(x) \varphi (\lambda x)\, x \cdot \nabla u(x) \dif x\\
  &= \int_{\R^N} \varphi (\lambda x)\, x \cdot \nabla \bigl(\tfrac{\abs{u}^2}{2}\bigr) (x) \dif x\\
  &= - \int_{\R^N}  \bigl( N \varphi (\lambda x)
                         + \lambda x \cdot \nabla \varphi (\lambda x) \bigr)
                    \frac{\abs{u(x)}^2}{2}
        \dif x.
\end{split}
\]
Lebesgue's dominated convergence theorem implies that
\[
 \lim_{\lambda \to 0} \int_{\R^N} u v_\lambda = - \frac{N}{2} \int_{\R^N} \abs{u}^2.
\]
Similarly, as
\(u \in W^{2, 2}_\mathrm{loc} (\R^N)\), the gradient term can be written as
\[
\begin{split}
 \int_{\R^N} \nabla u \cdot \nabla v_\lambda
  &= \int_{\R^N} \varphi (\lambda x)
                  \bigl(\abs{\nabla u}^2
                        + x \cdot \nabla \bigl(\tfrac{\abs{\nabla u}^2}{2}\bigr) (x) \bigr)
      \dif x\\
  &= - \int_{\R^N}  \bigl( (N - 2) \varphi (\lambda x)
                         + \lambda x \cdot \nabla \varphi (\lambda x) \bigr)
                    \frac{\abs{\nabla u(x)}^2}{2}
        \dif x.
\end{split}
\]
Lebesgue's dominated convergence again is applicable since \(\nabla u \in L^{2} (\R^N)\) and we obtain
\[
 \lim_{\lambda \to 0} \int_{\R^N} \nabla u \cdot \nabla v_\lambda = - \frac{N - 2}{2} \int_{\R^N} \abs{\nabla u}^2.
\]
The last term can be rewritten by integration by parts for every \(\lambda > 0\) as
\[
\begin{split}
 \int_{\R^N} \bigl(I_\alpha \ast F (u)\bigr) &( f(u) v_\lambda)
 = \int_{\R^N} \int_{\R^N} (F \circ u) (y) I_\alpha (x - y) \varphi (\lambda x)\, x \cdot \nabla (F \circ u ) (x) \dif x\dif y\\
 =\,& \frac{1}{2} \int_{\R^N} \int_{\R^N} I_\alpha (x - y) \Bigl((F \circ u)(y) \varphi (\lambda x) \,x \cdot \nabla (F \circ u) (x)\\
  & \qquad \qquad \qquad \qquad \qquad + (F \circ u) (x) \varphi (\lambda y) \,y \cdot \nabla (F \circ u) (y)\Bigr)\dif x\dif y\\
 =\,& - \int_{\R^N} \int_{\R^N} F \bigl(u(y)\bigr) I_\alpha (x - y) \bigl(N \varphi (\lambda x) + x \cdot \nabla \varphi (\lambda x)\bigr) F \bigl(u (x)\bigr) \dif x\dif y\\
  &\quad+ \frac{N - \alpha}{2} \int_{\R^N} \int_{\R^N} F \bigl(u (y)\bigr) I_\alpha (x - y)\\
  & \qquad \qquad \qquad \qquad  \qquad \frac{(x - y) \cdot \bigl(x \varphi (\lambda x) - y \varphi (\lambda y)\bigr)}{\abs{x - y}^2} F \bigl(u (x)\bigr)  \dif x\dif y.
\end{split}
\]
We can thus apply Lebesgue's dominated convergence theorem to conclude that
\[
 \lim_{\lambda \to 0} \int_{\R^N} \bigl(I_\alpha \ast F (u)\bigr) f (u)\, v_\lambda
= - \frac{N + \alpha}{2} \int_{\R^N} (I_\alpha \ast F (u)) F (u).\qedhere
\]
\end{proof}

\section{From solutions to groundstates}

\subsection{Solutions and paths}
\label{sectionPaths}
One of the application of the Poho\v zaev identity of the previous section is the possibility to associate to any variational solution of \eqref{problemMain} a path, following an argument of L.\thinspace Jeanjean and H.\thinspace Tanaka \cite{JeanjeanTanaka2003}.

\begin{proposition}[Lifting a solution to a path]
\label{propositionConstructPath}
If \(f \in C (\R; \R)\) satisfies \eqref{assumptWeakf} and \(u \in H^1 (\R^N) \setminus \{0\}\) solves \eqref{problemMain}, then there exists a path \(\gamma \in C\bigl([0, 1]; H^1 (\R^N)\bigr)\) such that
\begin{align*}
  \gamma (0) & = 0,\\
  \gamma (1/2) & = u,\\
  \mathcal{I} \bigl(\gamma (t)\bigr) & < \mathcal{I} (u),&  &\text{for every \(t \in [0, 1] \setminus \{1/2\}\)},\\
  \mathcal{I} \bigl(\gamma (1)\bigr) & < 0.
\end{align*}
\end{proposition}
\begin{proof}
The proof follows closely the arguments for the local problem developed by L.\thinspace Jeanjean and H.\thinspace Tanaka~\cite{JeanjeanTanaka2003}*{lemma 2.1}. We define the path \(\Tilde{\gamma} : [0, \infty) \to H^1 (\R^N)\) by
\[
 \Tilde{\gamma} (\tau) (x) =
 \begin{cases}
  u (x / \tau) & \text{if \(\tau > 0\)},\\
  0 & \text{if \(\tau = 0\)}.
 \end{cases}
\]
The function \(\Tilde{\gamma}\) is continuous on \((0, \infty)\); for every \(\tau > 0\),
\[
  \int_{\R^N} \abs{\nabla \Tilde{\gamma} (\tau)}^2 + \abs{\Tilde{\gamma} (\tau)}^2 =
  \tau^{N - 2} \int_{\R^N} \abs{\nabla u}^2 + \tau^N \int_{\R^N} \abs{u}^2,
\]
so that \(\Tilde{\gamma}\) is continuous at \(0\).
By the Poho\v zaev identity of proposition~\ref{propositionPohozhaev}, the functional can be computed for every \(\tau > 0\) as
\[
\begin{split}
 \mathcal{I} \bigl( \Tilde{\gamma} (\tau) \bigr) & = \frac{\tau^{N - 2}}{2} \int_{\R^N} \abs{\nabla u}^2 + \frac{\tau^N}{2} \int_{\R^N} \abs{u}^2 - \frac{\tau^{N + \alpha}}{2} \int_{\R^N} \bigl(I_\alpha \ast F (u)\bigr) F (u)\\ & = \Bigl(\frac{\tau^{N - 2}}{2} - \frac{(N - 2)\tau^{N + \alpha}}{2 (N + \alpha)}\Bigr)\int_{\R^N} \abs{\nabla u}^2 + \Bigl(\frac{\tau^N}{2} - \frac{N \tau^{N + \alpha}}{2 (N + \alpha)}\Bigr) \int_{\R^N} \abs{u}^2.
\end{split}
\]
It can be checked directly that \(\mathcal{I} \circ \Tilde{\gamma}\) achieves strict global maximum at \(1\): for every \(\tau \in [0, \infty) \setminus \{1\}\),
\(\mathcal{I} \bigl( \Tilde{\gamma} (\tau) \bigr) < \mathcal{I} (u)\). Since
\[
 \lim_{\tau \to \infty} \mathcal{I} \bigl(\Tilde{\gamma} (\tau) \bigr) = - \infty,
\]
the path \(\gamma\) can then be defined by a suitable change of variable.
\end{proof}

\subsection{Minimality of the energy and existence of a groundstate}
We now have all the tools available to show that the mountain-pass critical level \(b\) defined in \eqref{definitionb} coincides
with the ground state energy level \(c\) defined in \eqref{GroundStateC}, which completes the proof of theorem~\ref{theoremExistence}.

\begin{proof}[Proof of theorem~\ref{theoremExistence}]
By propositions \ref{propositionMinimax} and \ref{propositionPPS}, there exists a Poho\v zaev--Palais--Smale sequence \((u_n)_{n \in \N}\) in \(H^1 (\R^N)\) at the mountain-pass level \(b>0\), that converges weakly to some \(u \in H^1 (\R^N) \setminus\{0\}\) that solves \eqref{problemMain}. Since \(\lim_{n \to \infty} \mathcal{P} (u_n) = 0\), by the weak convergence of the sequence \((u_n)_{n \in \N}\), the weak lower-semicontinuity of the norm and the Poho\v zaev identity of proposition~\ref{propositionPohozhaev},
\begin{equation}
\label{ineqMinimal}
\begin{split}
 \mathcal{I} (u) &= \mathcal{I} (u) - \frac{\mathcal{P} (u)}{N + \alpha}\\
 & = \frac{\alpha + 2}{2 (N + \alpha)} \int_{\R^N} \abs{\nabla u}^2 + \frac{\alpha}{2 (N + \alpha)} \int_{\R^N} \abs{u}^2\\
 &\le \liminf_{n \to \infty} \frac{\alpha + 2}{2 (N + \alpha)} \int_{\R^N} \abs{\nabla u_n}^2 + \frac{\alpha}{2 (N + \alpha)} \int_{\R^N} \abs{u_n}^2\\
 & = \liminf_{n \to \infty} \mathcal{I} (u_n) - \frac{\mathcal{P} (u_n)}{N + \alpha}= \liminf_{n \to \infty} \mathcal{I} (u_n) = b.
\end{split}
\end{equation}
Since \(u\) is a nontrivial solution of \eqref{problemMain}, we have \(\mathcal{I} (u) \ge c\) by definition
of the ground state energy level \(c\), and hence \(c\le b\).

Let \(v \in H^1 (\R^N)\setminus\{0\}\) be another solution of \eqref{problemMain} such that \(\mathcal{I} (v) \le \mathcal{I} (u)\).
If we lift \(v\) to a path(proposition~\ref{propositionConstructPath}) and recall the definition \eqref{definitionb}
of the mountain-pass level \(b\), we conclude that \(\mathcal{I} (v) \ge b \ge \mathcal{I} (u)\).
We have thus proved that \(\mathcal{I} (v)=\mathcal{I} (u) = b = c\), and this concludes the proof of theorem~\ref{theoremExistence}.
\end{proof}

\subsection{Compactness of the set of groundstates}

As a byproduct of the proof of theorem~\ref{theoremExistence}, the weak convergence of the translated subsequence of proposition \ref{propositionPPS} can be upgraded into strong convergence.

\begin{corollary}[Strong convergence of translated Poho\v zaev--Palais--Smale sequences]
\label{corollaryStrong}
Under the assumptions of proposition~\ref{propositionPPS}, if
\[
 \liminf_{n \to \infty} \int_{\R^N} \abs{\nabla u_n}^2 + \abs{u_n}^2 > 0,
\]
and if
\[
 \liminf_{n \to \infty} \mathcal{I} (u_n) \le c,
\]
then there exists \(u \in H^1 (\R^N) \setminus \{0\}\) such that \(\mathcal{I}' (u) = 0\) and a sequence \((a_n)_{n \in \N}\) of points in \(\R^N\)  such that up to a subsequence \(u_n (\cdot - a_n) \to u\) strongly in \(H^1 (\R^N)\) as \(n \to \infty\).
\end{corollary}

\begin{proof}
By proposition~\ref{propositionPPS}, up to a subsequence and translations, we can assume that the sequence \((u_n)_{n \in \N}\) converges weakly to \(u\).
Since equality holds in \eqref{ineqMinimal},
\begin{multline*}
 \frac{\alpha + 2}{2 (N + \alpha)} \int_{\R^N} \abs{\nabla u}^2 + \frac{\alpha}{2 (N + \alpha)} \int_{\R^N} \abs{u}^2\\
 = \liminf_{n \to \infty} \frac{\alpha + 2}{2 (N + \alpha)} \int_{\R^N} \abs{\nabla u_n}^2 + \frac{\alpha}{2 (N + \alpha)} \int_{\R^N} \abs{u_n}^2,
\end{multline*}
and hence \((u_n)_{n \in \N}\) converges strongly to \(u\) in \(H^1 (\R^N)\).
\end{proof}

As a direct consequence we have some information on the set of groundstates:

\begin{proposition}[Compactness of the set of groundstates]
\label{propositionCompactnessGroundstates}
The set of groundstates
\[
 \mathcal{S}_c = \bigl\{ u \in H^1 (\R^N) \st \mathcal{I} (u) = c \text{ and \(u\) is a weak solution of \eqref{problemMain}} \bigr\}
\]
is compact in \(H^1 (\R^N)\) endowed with the strong topology up to translations in \(\R^N\).
\end{proposition}
\begin{proof}
This is a direct consequence of proposition~\ref{propositionPohozhaev} and corollary~\ref{corollaryStrong}.
\end{proof}

\begin{remark}[Uniform regularity of groundstates]
By the uniform regularity of solutions (remark~\ref{remarkUniformRegularity}) and the compactness of the set of groundstates (proposition~\ref{propositionCompactnessGroundstates}), for every
\(p \in [2, \frac{N}{\alpha} \frac{2 N}{N - 2})\), \(\mathcal{S}_c\) is bounded in \(L^p (\R^N)\).
\end{remark}

\section{Qualitative properties of groundstates}

\label{sectionSymmetry}

\subsection{Paths achieving the mountain pass level}
Arguments in this sections will use the following elementary property of the pathes in the construction of the mountain-pass critical level \(b\).

\begin{lemma}[Optimal paths yield critical points]
\label{lemmaOptimalPathCritical}
Let \(f \in C (\R; \R)\) satisfy \eqref{assumptWeakf} and
\(\gamma \in \Gamma\), where \(\Gamma\) is defined in \eqref{definitionGamma}.
If for every \(t \in [0, 1] \setminus \{t_*\}\), one has
\[
 b = \mathcal{I} \bigl(\gamma (t_*)\bigr) > \mathcal{I} \bigl(\gamma (t)\bigr),
\]
then \(\mathcal{I}' \bigl(\gamma (t_*)\bigr) = 0\).
\end{lemma}

\begin{proof}
This can be deduced from the quantitative deformation lemma of M.\thinspace Willem (see \cite{WillemMinimax}*{lemma~2.3}). Assume that \(\mathcal{I}' (\gamma (t_*)) \ne 0\). By continuity, it is possible to choose \(\delta > 0\) and \(\epsilon > 0\) such that \(\inf \{ \norm{\mathcal{I}' (v)} \st \norm{v - \gamma (t_*)} \le \delta \}> 8 \epsilon / \delta\).
With Willem's notations, take \(X = H^1 (\R^N)\), \(S = \{\gamma (t_*)\}\) and \(c = b\). By the deformation lemma, there exists \(\eta \in C ([0, 1]; H^1 (\R^N))\) such that \(\eta (1, \gamma) \in \Gamma\) and \(\mathcal{I}\bigl(\eta (1, \gamma (t_*))\bigr) \le b - \epsilon < b\) and for every \(t \in [0, 1]\), we have \(\mathcal{I}\bigl(\eta (1, \gamma (t))\bigr) \le \mathcal{I}\bigl( \gamma (t)\bigr) < b\). Since \([0, 1]\) is compact, we conclude with the contradiction that \(\sup_{t \in [0, 1]} \mathcal{I} \bigl(\eta (1, \gamma (t))\bigr) < b\).
\end{proof}

\subsection{Positivity of groundstates}
We now prove that when \(f\) is odd, groundstates do not change sign.

\begin{proposition}[Groundstates do not change sign]
\label{propositionPositive}
Let \(f \in C (\R; \R)\) satisfy \eqref{assumptWeakf}.
If \(f\) is odd and does not change sign on \((0, \infty)\),
then any groundstate \(u \in H^1 (\R^N)\) of \eqref{problemMain} has constant sign.
\end{proposition}
\begin{proof}
Without loss of generality, we can assume that \(f \ge 0\) on \((0, \infty)\).
By proposition~\ref{propositionConstructPath}, there exists an optimal path \(\gamma \in \Gamma\) on which the functional \(\mathcal{I}\)
achieves its maximum at \(1/2\).
Since \(f\) is odd, \(F\) is even and thus for every \(v \in H^1 (\R^N)\).
\[
  \mathcal{I} (\abs{v}) = \mathcal{I} (v).
\]
Hence, for every \(t \in [0, 1] \setminus \{1/2\}\),
\[
\mathcal{I} (\abs{\gamma (t)}) = \mathcal{I} (\gamma (t)) =  \mathcal{I} (\gamma (\tfrac{1}{2})) = \mathcal{I} (\abs{\gamma (\tfrac{1}{2})}).
\]
By lemma~\ref{lemmaOptimalPathCritical}, \(\abs{u} = \abs{\gamma (1/2)}\) is also a groundstate. It satisfies the equation
\[
 - \Delta \abs{u} + \abs{u} = \bigl(I_\alpha \ast F (\abs{u})\bigr) f (\abs{u}).
\]
Since \(u\) is continuous by proposition~\ref{propositionRegularity}, by the strong maximum principle we conclude that \(\abs{u} > 0\) on \(\R^N\)
and thus \(u\) has constant sign.
\end{proof}

\subsection{Symmetry of groundstates}
In this section, we now prove that groundstates are radial.

\begin{proposition}[Groundstates are symmetric]
\label{propositionSymmetry}
Let \(f \in C (\R; \R)\) satisfies \eqref{assumptWeakf}.
If \(f\) is odd and does not change sign on \((0, \infty)\), then any groundstate \(u \in H^1 (\R^N)\) of \eqref{problemMain}
is radially symmetric about a point.
\end{proposition}

The argument relies on polarizations. It is intermediate between the argument based on equality cases in polarization inequalities \cite{MorozVanSchaftingenGround} and the argument based on the Euler-Lagrange equation satisfied by polarizations \citelist{\cite{BartschWethWillem2005}\cite{VanSchaftingenWillem2008}}.

Before proving proposition~\ref{propositionSymmetry}, we recall some elements of the theory of polarization \citelist{\cite{Baernstein1994}\cite{VanSchaftingenWillem2004}\cite{BrockSolynin2000}\cite{Willem2007}}. Assume that  \(H \subset \R^N\) is a closed half-space and that \(\sigma_H\) is the reflection with respect to \(\partial H\). The  polarization \(u^H : \R^N\to \R\) of \(u : \R^N \to \R\) is defined for \(x \in \R^N\) by
\[
  u^H (x)
= \begin{cases}
    \max \bigl(u(x), u (\sigma_H (x)\bigr) & \text{if \(x \in H\)},\\
    \min \bigl(u(x), u (\sigma_H (x)\bigr) & \text{if \(x \not \in H\)}.
  \end{cases}
\]

We will use the following standard property of polarizations \cite{BrockSolynin2000}*{lemma 5.3}.
\begin{lemma}[Polarization and Dirichlet integrals]
\label{lemmaPolarizationGradient}
If \(u \in H^1 (\R^N)\), then \(u^H \in H^1 (\R^N)\) and
\[
\int_{\R^N} \abs{\nabla u^H}^2
= \int_{\R^N} \abs{\nabla u}^2.
\]
\end{lemma}

We shall also use a polarization inequality with equality cases \cite{MorozVanSchaftingenGround}*{lemma~5.3}
(without the equality cases, see \citelist{\cite{Baernstein1994}*{corollary 4}\cite{VanSchaftingenWillem2004}*{proposition 8}}).

\begin{lemma}[Polarization and nonlocal integrals]
\label{lemmaEqualityMaster}
Let \(\alpha \in (0, N)\), \(u \in L^{\frac{2 N}{N + \alpha}} (\R^N)\) and  \(H \subset \R^N\) be a closed half-space.
If \(u \ge 0\), then
\[
\int_{\R^N} \int_{\R^N} \frac{u(x)\, u(y)}{\abs{x - y}^{N - \alpha}} \,dx\,dy\\
\le \int_{\R^N} \int_{\R^N} \frac{ u^H(x)\, u^H(y)}{\abs{x - y}^{N - \alpha}} \,dx\,dy,
\]
with equality if and only if either \(u^H = u\) or \(u^H = u \circ \sigma_H\).
\end{lemma}

The last tool that we need is a characterization of symmetric functions by polarizations
\citelist{\cite{VanSchaftingenWillem2008}*{proposition 3.15}\cite{MorozVanSchaftingenGround}*{lemma 5.4}}

\begin{lemma}[Symmetry and polarization]
\label{lemmaSymmetry}
Assume that \(u \in L^2 (\R^N)\) is nonnegative. There exist \(x_0 \in \R^N\) and a nonincreasing function \(v : (0, \infty) \to \R\)
such that for almost every \(x \in \R^N\), \(u (x) = v (\abs{x - x_0})\) if and only if
for every closed half-space \(H \subset \R^N\), \(u^H = u\) or \(u^H = u \circ \sigma_H\).
\end{lemma}

\begin{proof}[Proof of proposition~\ref{propositionSymmetry}]
Without loss of generality, we can assume that \(f \ge 0\) on \((0, \infty)\).
By proposition~\ref{propositionPositive}, we can assume that \(u > 0\).
By proposition~\ref{propositionConstructPath}, there exists an optimal path \(\gamma\) such that \(\gamma (1/2) = u\) and \(\gamma (t) \ge 0\) for every \(t \in [0, 1]\). For every half-space \(H\) define the path \(\gamma^H \in [0, 1] \to H^1 (\R^N)\) by \(\gamma^H (t) = (\gamma (t))^H\). By lemma~\ref{lemmaPolarizationGradient}, \(\gamma^H \in C([0, 1]; H^1 (\R^N))\). Observe that since \(F\) is nondecreasing, \(F (u^H) = F (u)^H\), and therefore, for every \(t \in [0, 1]\), by lemmas~\ref{lemmaPolarizationGradient} and \ref{lemmaEqualityMaster},
\[
 \mathcal{I} (\gamma^H (t)) \le \mathcal{I} (\gamma (t)).
\]
Observe that \(\gamma^H \in \Gamma\) so that
\[
\max_{t \in [0, 1]} \mathcal{I}(\gamma^H (t)) \ge b.
\]
Since for every \(t \in [0, 1] \setminus \{1/2\}\),
\[
 \mathcal{I} (\gamma^H (t)) \le \mathcal{I} (\gamma (t)) < b,
\]
we have
\[
 \mathcal{I} (\gamma^H (\tfrac{1}{2})) = \mathcal{I} (\gamma (\tfrac{1}{2})) = b.
\]
By lemmas~\ref{lemmaPolarizationGradient} and \ref{lemmaEqualityMaster}, we have either \(F (u)^H = F (u)\) or \(F (u^H) = F (u \circ \sigma_H)\) in \(\R^N\).
Assume that \(F (u)^H = F (u)\).
Then, we have for every \(x \in H\),
\[
  \int_{u (\sigma_H (x) )}^{u(x)} f (s) \dif s = F (u (x)) - F (u (\sigma_H (x)))
\ge 0;
\]
this implies that either \(u (\sigma_H (x) ) \le u (x)\) or \(f = 0\) on \([u (x), u (\sigma_H (x) )]\). In particular, \(f (u^H) = f (u)\) on \(\R^N\).
By lemma~\ref{lemmaOptimalPathCritical} applied to \(\gamma^H\), we have \(\mathcal{I}' (u^H) = 0\); and therefore,
\[
  - \Delta u^H + u^H = \bigl(I_\alpha \ast F (u^H)\bigr) f (u^H)
= \bigl(I_\alpha \ast F (u)\bigr) f (u).
\]
Since \(u\) satisfies \eqref{problemMain}, we conclude that
\(u^H = u\).

If \(F (u^H) = F (u \circ \sigma_H)\), we conclude similarly that \(u^H = u \circ \sigma_H\).
Since this holds for arbitrary \(H\), we conclude by lemma~\ref{lemmaSymmetry} that \(u\) is radial and radially decreasing.
\end{proof}

\section{Alternative proof of the existence}

\label{sectionSymmetricMountainPass}
In this section we sketch an alternative proof of the existence of a nontrivial solution \(u \in H^1 (\R^N) \setminus \{0\}\) such that \(c \le \mathcal{I} (u) \le b\), under the additional symmetry assumption of theorem~\ref{theoremSymmetry} and in the spirit of
the symmetrization arguments of H.\thinspace Berestycki and P.-L.\thinspace Lions \cite{BerestyckiLions1983}*{pp.~325-326}.
The advantage of this approach is that it bypasses the concentration compactness argument
and delays the Poho\v zaev identity which is still needed to prove that \(b \le c\).

\begin{proof}[Proof of theorem~\ref{theoremExistence} under the additional assumptions of theorem~\ref{theoremSymmetry}]
In addition to \eqref{assumptWeakf}, \eqref{assumptStrongf} and \eqref{assumptNontrivialf},
assume that \(f\) is an odd function which has constant sign on \((0, \infty)\).
With this additional assumption,
\[
 \mathcal{I} \circ \Phi(\sigma, \abs{v}^H) \le \mathcal{I}\circ \Phi (\sigma, v).
\]
Therefore, by the symmetric variational principle \cite{VanSchaftingen2005}*{theorem 3.2}, we can prove as in the proof of proposition~\ref{propositionMinimax} the existence of a sequences \((u_n)_{n \in \N}\) and \((v_n)_{n \in \N}\) such that as \(n \to \infty\),
\begin{align*}
\mathcal{I} (u_n) & \to 0,\\
  \mathcal{I}' (u_n) &\to 0 & & \text{in \(\bigl(H^1 (\R^N)\bigr)'\)},\\
  \mathcal{P} (u_n) &\to 0,\\
  u_n - v_n &\to 0 & & \text{in \(L^2 (\R^N) \cap L^\frac{2 N}{N - 2} (\R^N)\)},
\end{align*}
and \(v_n\) is radial for every \(n \in \N\).

As previously, the sequence \((u_n)_{n \in \N}\) is bounded in \(H^1 (\R^N)\); by the P\'olya--Szeg\H o inequality, the sequence \((v_n)_{n \in \N}\) is also bounded in \(H^1 (\R^N)\). Since \(v_n\) is radial for every \(n \in \N\), the sequence \((v_n)_{n \in \N}\) is compact in \(L^p (\R^N)\) for every \(p \in (2, \frac{2 N}{N - 2})\) \cite{WillemMinimax}*{Corollary 1.26}.
 As \(u_n - v_n \to 0\) in \(L^2 (\R^N) \cap L^\frac{2 N}{N - 2} (\R^N)\), the sequence \((u_n)_{n \in \N}\) is also compact \(L^p (\R^N)\) for every \(p \in (2, \frac{2 N}{N - 2})\).

In view of \eqref{assumptStrongf}, this implies that
\(
  F (u_n) \to F (u)
\) as \(n \to \infty\) in \(L^\frac{2 N}{N + \alpha} (\R^N)\) and thus
\begin{equation}
\label{eqLimitNonlinear}
 \lim_{n \to \infty} \int_{\R^N} \bigl(I_\alpha \ast F (u_n)\bigr) F (u_n)=\int_{\R^N} \bigl(I_\alpha \ast F (u)\bigr) F (u) > 0.
\end{equation}
Now one can prove than that \(u_n (\cdot - a_n)\) converges to a nontrivial solution \(u \in H^1 (\R^N) \setminus \{0\}\)
as in the proof of proposition~\ref{propositionPPS}. By \eqref{eqLimitNonlinear}, it also follows that
\[
 c \le \mathcal{I} (u) \le b.
\]
Finally, employing the Poho\v zaev identity as in the proof of theorem~\ref{theoremExistence} allows to conclude that \(c = b\).
\end{proof}

\end{document}